\documentclass[12pt,reqno]{amsart}
\usepackage{graphicx}
\newtheorem{thm}{Theorem}[section]

\theoremstyle{definition}

\theoremstyle{remark}
\newtheorem{rem}[thm]{Remark}
\numberwithin{equation}{section}
\begin{document}
\title[ the inverse of the polygamma functions]{Inequalities for the inverses of the polygamma functions}%
\author{necdet batir}%
\address{department of mathematics\\
faculty of sciences and arts\\
nev{\c{s}}eh{\i}r hac{\i} bekta{\c{s}} veli university, nev{\c{s}}eh\i r, turkey}
\email{nbatir@hotmail.com}%
\email{nbatir@nevsehir.edu.tr}
%\thanks{}%
%\date{}%
%\dedicatory{}%
%\commby{}%
% ----------------------------------------------------------------
\begin{abstract}
We provide an elementary proof of the left side inequality and improve the right inequality in
\begin{align*}
\bigg[\frac{n!}{x-(x^{-1/n}+\alpha)^{-n}}\bigg]^{\frac{1}{n+1}}&<((-1)^{n-1}\psi^{(n)})^{-1}(x)\nonumber\\
&<\bigg[\frac{n!}{x-(x^{-1/n}+\beta)^{-n}}\bigg]^{\frac{1}{n+1}},
\end{align*}
where $\alpha=[(n-1)!]^{-1/n}$ and $\beta=[n!\zeta(n+1)]^{-1/n}$, which was proved in \cite{6}, and we   prove the following inequalities for the inverse of the digamma function $\psi$.
\begin{equation*}
\frac{1}{\log(1+e^{-x})}<\psi^{-1}(x)< e^{x}+\frac{1}{2}, \quad x\in\mathbb{R}.
\end{equation*}
The proofs are based on  nice applications of the mean value theorem for differentiation and elementary properties of the polygamma functions.
\end{abstract}
\maketitle
% ----------------------------------------------------------------
\section{introduction}
As it is well known for a positive real number $x$ the gamma function $\Gamma$ is defined to be
$$
\Gamma (x)=\int_0^\infty {u^{x - 1}e^{ - u}du}.
$$
It is a common knowledge that the gamma function plays a special role in the theory of special functions. The most important function related to the gamma function is the digamma or psi function $\psi(x)$, which is defined by  logarithmic derivative of the gamma function $\Gamma(x)$, that is,
$$
\psi(x)=\frac{\Gamma'(x)}{\Gamma(x)}.
$$
The digamma function $\psi$ is closely related with the Euler-Mascheroni constant $\gamma$ $(=0.57721...)$ and harmonic numbers $H_n=1+\frac{1}{2}+\cdots+\frac{1}{n}$. They satisfy $\psi(n+1)=-\gamma+H_n$  $(n\in\mathbb{N})$. In \cite{14} it is proved that
$\psi(p_n) \sim \log n$, when $n\sim\infty$, is equivalent with the Prime Number Theorem, where $p_n$ is $n$ th prime number. Functions $\psi'(x),\psi''(x),\psi'''(x),...$ are called polygamma functions in the literature. Polygamma functions are also very important functions and appear in the evaluations of many series and integrals \cite{2,9,12,13,16,18}. They are also related with many special functions such as the Riemann zeta function, Hurwitz zeta function, Clausen's function and generalized harmonic numbers. There exists a huge literature on inequalities for the digamma and polygamma functions, see \cite{3,4,6,7,10,11,15},  but for their inverses almost no inequality exists. The only known such an inequality is the following one which is  due to the author \cite[Theorem 2.5]{6}.
\begin{align}\label{e:1}
\bigg[\frac{n!}{x-(x^{-1/n}+\alpha)^{-n}}\bigg]^{\frac{1}{n+1}}&<\left((-1)^{n-1}\psi^{(n)}\right)^{-1}(x)\nonumber\\
&<\bigg[\frac{n!}{x-(x^{-1/n}+\beta)^{-n}}\bigg]^{\frac{1}{n+1}},
\end{align}
where $\alpha=[(n-1)!]^{-1/n}$ and $\beta=[n!\zeta(n+1)]^{-1/n}$.
Our first aim in this paper is to give an elementary proof of the left side of this inequality and to improve its right inequality.

Our second aim is to establish, through the use of elementary properties of polygamma functions and the mean value theorem, simple bounds for  the inverse of the digamma function $\psi$. Numerical experiments show that our  bounds are  are remarkably accurate for all $x\in\mathbb{R}$. In our proof we make use of the following relations for the gamma and  polygamma functions.
The gamma function satisfies the functional equation $\Gamma(x+1)=x\Gamma(x)$, and has the following canonical product representation
\begin{equation*}
\Gamma(x+1)=e^{-\gamma x}\prod_{k=1}^{\infty}\frac{k}{k+x}e^{x/k}\quad  x>-1,
\end{equation*}
where $\gamma=0.57721...$ is Euler-Mascheroni constant; see \cite[pg.346]{18}.
Taking logarithm of both sides of this formula, we obtain for $x>-1$
\begin{equation}\label{e:2}
\log\Gamma(x+1)=-\gamma x+\sum\limits_{k=1}^\infty\left[\frac{x}{k}-\log(x+k)+\log k\right].
\end{equation}
Differentiation gives
\begin{equation}\label{e:3}
\psi(x+1)=-\gamma+\sum\limits_{k=1}^\infty\left[\frac{1}{k}-\frac{1}{k+x}\right] \quad x>-1.
\end{equation}
For $x>0$ and $n=1,2,...$
\begin{equation}\label{e:4}
(-1)^{n-1}\psi^{(n)}(x)=\int\limits_{0}^\infty\frac{t^ne^{-xt}dt}{1-e^{-t}}=n!\sum\limits_{k=0}^{\infty}\frac{1}{(x+k)^{n+1}},
\end{equation}
and
\begin{equation}\label{e:5}
\psi^{(n)}(x+1)-\psi^{(n)}(x)=\frac{(-1)^nn!}{x^{n+1}}.
\end{equation}
See  \cite[p.260]{1} for these and further properties of these functions.
\section{main result}
We collect our main results in this section.
\begin{thm}For $x>0$ and $n=1,2,3,...$ we have
\begin{align}\label{e:6}
&\bigg[\frac{1}{n}\bigg(\frac{x}{(n-1)!}-\left(\left(\frac{(n-1)!}{x}\right)^{1/n}+1\right)^{-n}\bigg)\bigg]^{-\frac{1}{n+1}}\nonumber\\
&<\left((-1)^{n-1}\psi^{(n)}\right)^{-1}(x)<\left(\frac{(n-1)!}{x}\right)^{\frac{1}{n}}+\frac{1}{2}
\end{align}
\end{thm}
\begin{proof}
By (\ref{e:5}) we have for $n=1,2,3,...$
\begin{equation}\label{e:7}
\frac{(-1)^{n-1}(n-1)!}{x^{n}}=\psi^{(n-1)}(x+1)-\psi^{(n-1)}(x).
\end{equation}
Using  (\ref{e:4}), we get for $t>0$
\begin{equation}\label{e:8}
-\frac{1}{t^{n}}=\sum\limits_{k=0}^\infty\bigg[\frac{1}{(k+t+1)^{n}}-\frac{1}{(k+t)^n}\bigg].
\end{equation}
By mean value theorem for differentiation we have
\begin{equation}\label{e:9}
\frac{1}{(k+t+1)^{n}}-\frac{1}{(k+t)^n}=\frac{-n}{(k+\epsilon(k))^{n+1}},\quad t<\epsilon(k)<t+1.
\end{equation}
We therefore can write (\ref{e:8}) as following
\begin{equation}\label{e:10}
\frac{1}{nt^{n}}=\sum\limits_{k=0}^\infty\frac{1}{(k+\epsilon(k))^{n+1}}.
\end{equation}
From (\ref{e:9}) we have
\begin{equation}\label{e:11}
\epsilon(k)=\left[\frac{1}{n}\left(\frac{1}{(k+t)^{n}}-\frac{1}{(k+t+1)^{n}}\right)\right]^{-\frac{1}{n+1}}-k.
\end{equation}
We want to show that $\epsilon$ is strictly increasing on $(0,\infty)$. For this purpose we define
\begin{equation*}
f(u)=\left[\frac{1}{n}\left(\frac{1}{u^n}-\frac{1}{(u+1)^n}\right)\right]-u.
\end{equation*}
Clearly, $\epsilon'(k)>0$ is equivalent to $f'(u)>0$.
Differentiation gives
\begin{equation*}
f'(u)=\frac{1}{n+1}\left(u^{-n}-(u+1)^{-n}\right)\left(\frac{u^{-n}-(u+1)^{-n}}{n}\right)^{-\frac{n+2}{n+1}}-1.
\end{equation*}
We conclude that $\epsilon$ is strictly increasing if and only if
\begin{equation}\label{e:12}
\left[\frac{u^{-(n+1)}-(u+1)^{-(n+1)}}{n+1}\right]^{-\frac{1}{n+2}}<\left[\frac{u^{-n}-(u+1)^{-n}}{n}\right]^{-\frac{1}{n+1}},
\end{equation}
which follows from the fact that the generalized mean
$$
S_p(a,b)=\left[\frac{a^p-b^p}{p(a-b)}\right]^{1/(p-1)}
$$
is strictly increasing in $p$; see \cite[pg. 234]{17}. Indeed (\ref{e:12}) is equivalent to $S_{-(n+1)}(u,u+1)<S_{-n}(u,u+1)$.
Using (\ref{e:11}) we get
\begin{equation*}
\lim\limits_{k\to\infty}\epsilon(k)=t+\lim\limits_{u\to\infty}\bigg[\frac{1}{n}\left(u^{-n}-(u+1)^{-n}\right)\bigg]^{-\frac{1}{n+1}}-u.
\end{equation*}
In \cite[Lemma 1.4]{6} with $(t,s)=(1,0)$ this limit is evaluated and equals to 1/2. Thus,
\begin{equation}\label{e:13}
\epsilon(\infty):=\lim\limits_{k\to\infty}\epsilon(k)=t+\frac{1}{2}.
\end{equation}
By (\ref{e:11}) it is clear that
\begin{equation}\label{e:14}
\epsilon(0)=\bigg[\frac{1}{n}\left(t^{-n}-(t+1)^{-n}\right)\bigg]^{-\frac{1}{n+1}}.
\end{equation}
From the fact that $\epsilon$ is strictly increasing  we conclude from (\ref{e:10}) that
\begin{equation*}
\sum\limits_{k=0}^\infty\frac{1}{(k+\epsilon(\infty))^{n+1}}<\frac{1}{nt^n}<\sum\limits_{k=0}^\infty\frac{1}{(k+\epsilon(0))^{n+1}}
\end{equation*}
or taking into account (\ref{e:4})
\begin{equation}\label{e:15}
(-1)^{n-1}\psi^{(n)}(\epsilon(\infty))<\frac{(n-1)!}{t^n}<(-1)^{n-1}\psi^{(n)}(\epsilon(0)).
\end{equation}
Since the mapping $t\to(-1)^{n-1}\psi^{(n)}(t)$ is strictly decreasing, applying the inverse of this function to both sides of (\ref{e:15}) and using (\ref{e:13}) and (\ref{e:14}) we get

\begin{align}\label{e:16}
&\bigg[\frac{1}{n}\bigg(\frac{t}{(n-1)!}-\left(\left(\frac{(n-1)!}{t}\right)^{1/n}+1\right)\bigg]^{-\frac{1}{n+1}}\nonumber\\
&<\left((-1)^{n-1}\psi^{(n)}\right)^{-1}\left(\frac{(n-1)!}{t^n}\right)<t+\frac{1}{2}
\end{align}
where $[(-1)^{n-1}\psi^{(n)}]^{-1}$ is the iverse of the mapping $t\to(-1)^{n-1}\psi^{(n)}(t)$.
Setting $x=\frac{(n-1)!}{t^n}$ here we get the desired result (\ref{e:6}).
\end{proof}
\begin{rem}
Numerical computations show that the upper bound given in (\ref{e:6}) is much accurate than that of (\ref{e:1}).
\end{rem}
\begin{thm} For $x\in\mathbb{R}$ we have
\begin{equation}\label{e:17}
\frac{1}{\log(1+e^{-x})}<\psi^{-1}(x)< e^{x}+\frac{1}{2}.
\end{equation}
\end{thm}
\begin{proof} We want to give two different proofs.\\

\textbf{\textit{First proof.}} Applying the mean value theorem to $\log (\Gamma(t))$ on $[x,x+1]$ and using the functional equation $\Gamma(x+1)=x\Gamma(x)$ for the gamma function, we obtain
\begin{equation}\label{e:18}
\log x=\psi(x+\xi(x)),\quad 0<\xi(x)<1.
\end{equation}
We want to show that $\xi$ is strictly increasing on $(0,\infty)$. Differentiation gives
\begin{equation}\label{e:19}
\frac{1}{x}=(1+\xi'(x))\psi'(x+\xi(x))
\end{equation}
and
\begin{equation}\label{e:20}
-\frac{1}{x^2}=\xi''(x)\psi'(x+\xi(x))+(1+\xi'(x))^2\psi''(x+\xi(x)).
\end{equation}
By (\ref{e:19}) we have
$$
(1+\xi'(x))^2=\frac{1}{x^2\psi'(x+\xi(x))^2}.
$$
Substituting this into (\ref{e:20}) gives
\begin{equation}\label{e:21}
-x^2\xi''(x)\left[\psi'(x+\xi(x))\right]^3=\psi''(x+\xi(x))+\left[\psi'(x+\xi(x))\right]^2.
\end{equation}
We now show that the right hand side of this identity is positive. Let's define for $x>0$
$$
g(x)=\psi''(x)+\left(\psi'(x)\right)^2.
$$
Applying the recurrence relations
\begin{equation*}
\psi'(x+1)-\psi'(x)=-\frac{1}{x^2}\quad \mbox{and}\quad \psi''(x+1)-\psi''(x)=\frac{2}{x^3},
\end{equation*}
which follows from (\ref{e:5}), we obtain for positive $x$
\begin{equation}\label{e:22}
g(x)-g(x+1)=\frac{2}{x^2}\left(\psi'(x)-\frac{1}{x}-\frac{1}{2x^2}\right).
\end{equation}
Using
\begin{equation*}
\frac{1}{x}=\int\limits_{0}^\infty e^{-xt}dt,\quad \frac{1}{x^2}=\int\limits_{0}^\infty te^{-xt}dt\quad \mbox{and}\quad \psi'(x)=\int\limits_{0}^\infty \frac{te^{-xt}}{1-e^{-t}}dt,
\end{equation*}
we find that
\begin{equation}\label{e:23}
\psi'(x)-\frac{1}{x}-\frac{1}{2x^2}=\int\limits_{0}^\infty \frac{\alpha(t)e^{-xt}}{e^t-1}dt
\end{equation}
where $\alpha(t)=t+2+(t-2)e^t$. Since $\alpha(0)=\alpha'(0)=0$ and $\alpha''(t)=te^t>0$, we get $\alpha(t)>0$ for all $t>0$. Thus, the left side of (\ref{e:23}) is positive. This fact and (\ref{e:22}) together  imply   for $x>0$ that $g(x)-g(x+1)>0$. This reveals that
$$
g(x)>g(x+1)>g(x+2)>\cdots>f(x+n)>\lim\limits_{n\to\infty}g(x+n)=0.
$$
Since $\psi'(x)>0$ for $x>0$, we conclude from (\ref{e:21}) that $\xi''(x)<0$. That is, $\xi'$ is strictly decreasing on $(0,\infty)$. Since $\psi'$ is strictly decreasing and   $0<\xi(x)<1$,  (\ref{e:19}) yields
\begin{equation}\label{e:24}
\frac{1}{x\psi'(x)}-1<\xi'(x)<\frac{1}{x\psi'(x+1)}-1.
\end{equation}
Using the asymptotic expansion
\begin{equation*}
\psi'(x)\sim\frac{1}{x}+\frac{1}{2x^2}+\frac{1}{2x^3}+\cdots
\end{equation*}
see \cite[pg 260; 6.4.12]{1} we see that the limits of  both bounds in (\ref{e:24}) tend to 0 as $x$ goes to infinity, that is,  $\lim\limits_{x\to\infty}\xi'(x)=0$. We therefore have $\xi'(x)>\lim\limits_{x\to\infty}\xi'(x)=0$. Therefore, $\xi$ is a monotonic increasing  function of $x$ on $(0, \infty)$. Replacing $x$ by $x+1$ in (\ref{e:18}), and using the fact that both $\psi$ and $\xi$ are strictly increasing on $(0,\infty)$ we get
\begin{equation}\label{e:25}
\log (x+1)=\psi(x+\xi(x+1)+1)>\psi(x+\xi(x)+1).
\end{equation}
Employing the well known recurrence relation
\begin{equation}\label{e:26}
\psi(x+1)-\psi(x)=\frac{1}{x},
\end{equation}
we obtain from (\ref{e:25}) and (\ref{e:18})
\begin{equation*}
\log(x+1)>\frac{1}{x+\xi(x)}+\psi(x+\xi(x))=\frac{1}{x+\xi(x)}+\log x
\end{equation*}
or
\begin{equation*}
x+\xi(x)>\frac{1}{\log(1+1/x)}.
\end{equation*}
Applying $\psi$ both sides, this becomes by (\ref{e:18})
\begin{equation}\label{e:27}
\log x=\psi(x+\xi(x))>\psi\left(\frac{1}{\log(1+1/x)}\right).
\end{equation}
Replacing $x$ by $x+1$ in (\ref{e:18}) and using the relation (\ref{e:26}) we get that
\begin{equation}\label{e:28}
\xi(x+1)=\frac{1}{\log(x+1)-\psi(x+\xi(x+1))}-x.
\end{equation}
 Since $\xi$ is bounded and strictly increasing it has a limit as $x$ approaches to infinity. We therefore conclude from (\ref{e:18}) and (\ref{e:28}) that
\begin{align*}
\lim\limits_{x\to\infty}\xi(x+1)=&\lim\limits_{x\to\infty}\left[\frac{1}{\log(x+1)-\psi(x+\xi(x+1))}-x\right]\nonumber\\
=&\lim\limits_{x\to\infty}\left[\frac{1}{\log(x+1)-\psi(x+\xi(x))}-x\right]\nonumber\\
=&\lim\limits_{x\to\infty}\left[\frac{1}{\log(1+1/x)}-x\right]=\frac{1}{2}.
\end{align*}
So monotonic increase of $\xi$ and $\psi$  and Equation (\ref{e:28}) imply that
\begin{equation}\label{e:29}
\log x=\psi(x+\xi(x))<\psi\left(x+\lim\limits_{x\to\infty}\xi(x)\right)=\psi\left(x+\frac{1}{2}\right).
\end{equation}
Combining (\ref{e:27}) and (\ref{e:29}) it follows that
\begin{equation*}
\psi\left(\frac{1}{\log(1+1/x)}\right)<\log x<\psi\left(x+\frac{1}{2}\right).
\end{equation*}
The desired inequality (\ref{e:6}) now follows from replacing $x$ by $e^x$ here, after applying $\psi^{-1}$ both sides.\\

\textbf{\textit{Second proof.}} Utilizing the functional equation $\Gamma(x+1)=x\Gamma(x)$ in (\ref{e:4}) and using (\ref{e:2}), we obtain
\begin{align}\label{e:30}
\log x&=\log\Gamma(x+1)-\log\Gamma(x)\nonumber\\
&=-\gamma x+\sum\limits_{k=1}^\infty\bigg[\frac{x}{k}-\log(x+k)+\log k\bigg]\nonumber \\
&+\gamma (x-1)-\sum\limits_{k=1}^\infty\bigg[\frac{x-1}{k}-\log(x+k-1)+\log k\bigg]\nonumber\\
&=-\gamma+\sum\limits_{k=1}^\infty\bigg[\frac{1}{k}-\log(x+k)+\log(x+k-1)\bigg].
\end{align}
By mean value theorem we get
\begin{equation}\label{e:31}
\log(x+k)-\log(x+k-1)=\frac{1}{k+\phi(k)},\quad x-1<\phi(k)<x.
\end{equation}
So (\ref{e:30}) becames
\begin{equation}\label{e:32}
\log x=-\gamma +\sum\limits_{k=1}^\infty\bigg[\frac{1}{k}-\frac{1}{k+\phi(k)}\bigg]
\end{equation}
It is clear from (\ref{e:27})  that
\begin{equation}\label{e:33}
\phi(k)=\frac{1}{\log\bigg(1+\frac{1}{x+k-1}\bigg)}-k,
\end{equation}
and
\begin{equation}\label{e:34}
\phi(1)=\frac{1}{\log\bigg(1+\frac{1}{x}\bigg)}-1.
\end{equation}
We can easily compute that
\begin{equation}\label{e:35}
\phi(\infty):=\lim\limits_{k\to\infty}\phi(k)=x-\frac{1}{2}.
\end{equation}
We want to show that $\phi$ is strictly increasing on $[1,\infty)$. Setting $u=x+k-1$ in (\ref{e:33}) its right hand side becomes
\begin{equation*}
x-1+\frac{1}{\log\bigg(1+\frac{1}{u}\bigg)}-u=g(u),\,\mbox{say}.
\end{equation*}
So in order to prove that $\phi$ is strictly increasing on $[1,\infty)$, it suffices to show that $g$ is strictly increasing on $(0,\infty.)$ If we differentiate $g$ and apply the well known geometric-logarithmic mean inequality $G\leq L$ we see that $g'(u)>0$, which implies that $\phi$ is strictly increasing on $[1,\infty)$. Hence we conclude from (\ref{e:32}) that
\begin{equation*}
-\gamma+\sum\limits_{k=1}^\infty\left(\frac{1}{k}-\frac{1}{k+\phi(1)}\right)\leq\log x \leq-\gamma+\sum\limits_{k=1}^\infty\left(\frac{1}{k}-\frac{1}{k+\phi(\infty)}\right).
\end{equation*}
Using (\ref{e:32}), (\ref{e:33}) and (\ref{e:3}) we obtain
\begin{equation*}
\psi\left(\frac{1}{\log(1+1/x)}\right)\leq\log x\leq \psi\left(x+\frac{1}{2}\right).
\end{equation*}
Since $\psi^{-1}$ is strictly increasing on $(-\infty,\infty)$, applying $\psi^{-1}$ both sides of this inequalities we get
\begin{equation*}
\frac{1}{\log(1+1/x)}\leq\psi^{-1}(\log x)\leq x+\frac{1}{2}.
\end{equation*}
Replacing $x$ by  $e^x$ here completes the proof of Theorem 2.1.
\end{proof}
\bibliographystyle{amsplain}

\end{document}